\documentclass[english]{amsart}
\usepackage[T1]{fontenc}
\usepackage[utf8]{inputenc} 
\usepackage{amsmath}
\usepackage{amsthm}
\usepackage{amssymb} 
\usepackage{mathrsfs}
\usepackage{mathtools}
\usepackage{enumitem}
\usepackage[bookmarks]{hyperref}

\theoremstyle{plain}
\newtheorem{thm}{\protect\theoremname}[section]
  \theoremstyle{plain}
  \newtheorem{fact}[thm]{\protect\factname}
  \theoremstyle{definition}
  \newtheorem{defn}[thm]{\protect\definitionname}
  \theoremstyle{plain}
  
  \theoremstyle{plain}
  \newtheorem{prop}[thm]{\protect\propositionname}
  \theoremstyle{plain}
  \newtheorem{ex}[thm]{\protect\exname}
  \theoremstyle{plain}
  \newtheorem{cor}[thm]{\protect\corname}
  \theoremstyle{plain}
  
  \theoremstyle{plain}
  \newtheorem{quest}[thm]{\protect\questname}
  \theoremstyle{flaweddefinition}

\usepackage{babel}
  \providecommand{\definitionname}{Definition}
  \providecommand{\factname}{Fact}
  \providecommand{\lemmaname}{Lemma}
  \providecommand{\propositionname}{Proposition}
\providecommand{\theoremname}{Theorem}
\providecommand{\exname}{Example}
\providecommand{\corname}{Corollary}
\providecommand{\notaname}{Notation}
\providecommand{\questname}{Question}
\providecommand{\flaweddefinitionname}{Flawed Definition}

\newcommand{\e}{\varepsilon}

\providecommand{\dotdiv}{
  \mathbin{
    \vphantom{+}
    \mathrm{
      \mathsurround=0pt 
      \ooalign{
        \noalign{\kern-.35ex}
        \hidewidth$\smash{\cdot}$\hidewidth\cr 
        \noalign{\kern.35ex}
        $-$\cr 
      }%
    }%
  }%
}

\newcommand{\vertiii}[1]{{\left\vert\kern-0.25ex\left\vert\kern-0.25ex\left\vert #1 
    \right\vert\kern-0.25ex\right\vert\kern-0.25ex\right\vert}}

\begin{document}


\title{Indiscernible Subspaces and Minimal Wide Types}

\author{James Hanson}
\address{Department of Mathematics, University of Wisconsin--Madison, 480 Lincoln Dr., Madison WI 53706}

\date{\today}

\begin{abstract}
We develop the machinery of \emph{indiscernible subspaces} in continuous theories of expansions of Banach spaces, showing that any such theory has an indiscernible subspace and therefore an indiscernible set. We extend a result of Shelah and Usvyatsov \cite{SHELAH2019106738} by showing that a sequence of realizations of a (possibly unstable) minimal wide type $p$ is a Morley sequence in $p$ if and only if it is the orthonormal basis of an indiscernible subspace in $p$. We also give an example showing that minimal wide types do not generally have type-definable indiscernible subspaces (answering a question of Shelah and Usvyatsov \cite{SHELAH2019106738}), as well as an example showing that our result fails for non-minimal wide types, even in $\omega$-stable theories. 
\end{abstract}

\maketitle






\section{Introduction} 

Continuous first-order logic is a generalization of first-order logic suitable for studying metric structures, which are mathematical structures with an underlying complete metric and with uniformly continuous functions and $\mathbb{R}$-valued predicates. A rich class of metric structures arise from Banach spaces and expansions thereof. An active area of research in continuous logic is the characterization of inseparably categorical continuous theories. 
For a general introduction to continuous logic, see \cite{MTFMS}.

In the present work we will consider expansions of Banach spaces.  We introduce the notion of an indiscernible subspace. An indiscernible subspace is a subspace in which types of tuples of elements only depend on their quantifier-free type in the reduct consisting of only the metric and the constant $\mathbf{0}$. Similarly to indiscernible sequences, indiscernible subspaces are always consistent with a Banach theory (with no stability assumption, see Theorem \ref{thm:exist}), but are not always present in every model. We will show that an indiscernible subspace always takes the form of an isometrically embedded real Hilbert space wherein the type of any tuple only depends on its quantifier-free type in the Hilbert space. The notion of an indiscernible subspace is of independent interest in the model theory of Banach and Hilbert structures, and in particular here we use it to improve the results of Shelah and Usvyatsov in the context of types in the full language (as opposed to $\Delta$-types). Specifically, in this context we give a shorter proof of Shelah and Usvyatsov's main result \cite[Prop.\ 4.13]{SHELAH2019106738}, we improve their result on the strong uniqueness of Morley sequences in minimal wide types \cite[Prop.\ 4.12]{SHELAH2019106738}, and we expand on their commentary on the ``induced structure'' of the span of a Morley sequence in a minimal wide type \cite[Rem.\ 5.6]{SHELAH2019106738}. This more restricted case is what is relevant to inseparably categorical Banach theories, so our work is applicable to the problem of their characterization. 

Finally, we present some relevant counterexamples and in particular we resolve (in the negative) the question of Shelah and Usvyatsov presented at the end of Section 5 of \cite{SHELAH2019106738}, in which they ask whether or not the span of a Morley sequence in a minimal wide type is always a type-definable set.

\subsection{Background}

For $K \in \{\mathbb{R}, \mathbb{C}\}$, we think of a $K$-Banach space $X$ as being a metric structure $\mathfrak{X}$ whose underlying set is the closed unit ball $B(X)$ of $X$ with metric $d(x,y) = \left\lVert x - y \right\rVert$.\footnote{For another equivalent approach, see \cite{MTFMS}, which encodes Banach structures as many-sorted metric structures with balls of various radii as different sorts.}  This structure is taken to have for each tuple $\bar{a} \in K$  an $|\bar{a}|$-ary predicate $s_{\bar{a}}(\bar{x}) = \left\lVert \sum_{i<|\bar{a}|} a_i x_i \right\rVert$, although we will always write this in the more standard form. 
 Note that we evaluate this in $X$ even if $\sum_{i<|\bar{a}|} a_i x_i$ is not actually an element of the structure $\mathfrak{X}$.  For convenience, we will also have a constant for the zero vector, $\mathbf{0}$, and an $n$-ary function $\sigma_{\bar{a}}(\bar{x})$ such that $\sigma_{\bar{a}}(\bar{x}) = \sum_{i<|\bar{a}|} a_i x_i$ if it is in $B(X)$ and  $\sigma_{\bar{a}}(\bar{x}) = \frac{\sum_{i<|\bar{a}|} a_i x_i}{\left\lVert \sum_{i<|\bar{a}|} a_i x_i \right\rVert}$ otherwise. If $|a|\leq 1$, we will write $ax$ for $\sigma_{a}(x)$. Note that while this is an uncountable language, it is interdefinable with a countable reduct of it (restricting attention to rational elements of $K$). These structures capture the typical meaning of the ultraproduct of Banach spaces.  As is common, we will conflate $X$ and the metric structure $\mathfrak{X}$ in which we have encoded $X$.

\begin{defn}
A \emph{Banach (or Hilbert) structure} is a metric structure which is the expansion of a Banach (or Hilbert) space. A \emph{Banach (or Hilbert) theory} is the theory of such a structure. The adjectives \emph{real} and \emph{complex} refer to the scalar field $K$.
\end{defn}

$C^{\ast}$- and other Banach algebras are commonly studied examples of Banach structures that are not just Banach spaces.

A central problem in continuous logic is the characterization of inseparably categorical countable theories, that is to say countable theories with a unique model in each uncountable density character. The analog of Morley's theorem was shown in continuous logic via related formalisms \cite{ben-yaacov_2005, Shelah2011}, but no satisfactory analog of the Baldwin-Lachlan theorem or its precise structural characterization of uncountably categorical discrete theories in terms of strongly minimal sets is known. Some progress in the specific case of Banach theories has been made in \cite{SHELAH2019106738}, in which Shelah and Usvyatsov introduce the notion of a wide type and the notion of a minimal wide type, which they argue is the correct analog of strongly minimal types in the context of inseparably categorical Banach theories.

\begin{defn}
A type $p$ in a Banach theory is \emph{wide} if its set of realizations consistently contain the unit sphere of an infinite dimensional real subspace.

A type is \emph{minimal wide} if it is wide and has a unique wide extension to every set of parameters.
\end{defn}
In \cite{SHELAH2019106738}, Shelah and Usvyatsov were able to show that every Banach theory has wide complete types using the following classical concentration of measure results of Dvoretzky and Milman, which Shelah and Usvyatsov refer to as the Dvoretzky-Milman theorem.

\begin{fact}[Dvoretzky-Milman theorem] \label{fact:DM-thm}
Let $(X,\left\lVert \cdot \right\rVert)$ be an infinite dimensional real Banach space with unit sphere $S$ and let $f:S \rightarrow \mathbb{R}$ be a uniformly continuous function. For any $k<\omega$ and $\e > 0$, there exists a $k$-dimensional subspace $Y \subset X$ and a Euclidean norm\footnote{A norm $\vertiii{\cdot}$  is \emph{Euclidean} if it satisfies the parallelogram law, $2\vertiii{a}^2 + 2 \vertiii{b}^2 = \vertiii{a+b}^2 + \vertiii{a-b}^2,$ or equivalently if it is induced by an inner product.} $\vertiii{\cdot}$ on $Y$ such that for any $a,b \in S\cap Y$, we have $\vertiii{a} \leq \left\lVert a\right\rVert \leq (1 + \e)\vertiii{a}$ and $|f(a) - f(b)| < \e$.\footnote{Fact \ref{fact:DM-thm} without $f$ is (a form of) Dvoretsky's theorem.} 
\end{fact}


Shelah and Usvyatsov showed that in a stable Banach theory every wide type has a minimal wide extension (possibly over a larger set of parameters) and that every Morley sequence in a minimal wide type is an orthonormal basis of a subspace isometric to a real Hilbert space. Furthermore, they showed that in an inseparably categorical Banach theory, every inseparable model is prime over a countable set of parameters and a Morley sequence in some minimal wide type, analogously to how a model of a discrete uncountably categorical theory is always prime over some finite set of parameters and a Morley sequence in some strongly minimal type.

The key ingredient to our present work is the following result, due to Milman. 
 It extends the Dvoretzky-Milman theorem in a manner analogous to the extension of the pigeonhole principle by Ramsey's theorem.\footnote{The original Dvoretzky-Milman result is often compared to Ramsey's theorem, such as when Gromov coined the term \emph{the Ramsey-Dvoretzky-Milman phenomenon} \cite{gromov1983}, but in the context of Fact \ref{fact:main} it is hard not to think of the $n=1$ case as being analogous to the pigeonhole principle and the $n>1$ cases as being analogous to Ramsey's theorem.}

\begin{defn}\label{defn:main-defn}
Let $(X,\left\lVert \cdot \right\rVert)$ be a Banach space. If $a_0,a_1,\dots,a_{n-1}$ and $b_0,b_1,\dots,\allowbreak b_{n-1}$ are ordered $n$-tuples of elements of $X$, we say that $\bar{a}$ and $\bar{b}$ are \emph{congruent} if $\left\lVert a_i - a_j\right\rVert=\left\lVert b_i - b_j \right\rVert$ for all $ i,j \leq n$, where we take $a_{n}=b_{n}=\mathbf{0}$. We will write this as $\bar{a} \cong \bar{b}$.
\end{defn}

\begin{fact}[\cite{zbMATH03376472}, Thm.\ 3] \label{fact:main}
Let $S^\infty$ be the unit sphere of a separable infinite dimensional real Hilbert space ${H}$ and let $f:(S^\infty)^n \rightarrow \mathbb{R}$ be a uniformly continuous function. For any $\varepsilon>0$ and any $k<\omega$ there exists a $k$-dimensional subspace $V$ of $H$ such that for any $a_0,a_1,\dots,a_{n-1},b_0,b_1,\dots,b_{n-1}\in S^\infty$ with $\bar{a} \cong \bar{b}$, $|f(\bar{a})-f(\bar{b})| < \varepsilon$.
\end{fact}

Note that the analogous result for inseparable Hilbert spaces follows immediately, by restricting attention to a separable infinite dimensional subspace. Also note that by using Dvoretzky's theorem and an easy compactness argument, Fact~\ref{fact:main} can be generalized to arbitrary infinite dimensional Banach spaces. Also note that while Fact~\ref{fact:DM-thm} and Fact~\ref{fact:main} are stated for real Banach spaces, analogous statements for complex Banach spaces can be easily derived from them.

\subsection{Connection to Extreme Amenability}

 A modern proof of Fact \ref{fact:main} would go through the extreme amenability of the unitary group of an infinite dimensional Hilbert space endowed with the strong operator topology, or in other words the fact that any continuous action of this group on a compact Hausdorff space has a fixed point, which was originally shown in \cite{10.2307/2374298}.  This connection is unsurprising. It is well known that the extreme amenability of $\mathrm{Aut}(\mathbb{Q})$ (endowed with the topology of pointwise convergence) can be understood as a restatement of Ramsey's theorem. It is possible to use this to give a high brow proof of the existence of indiscernible sequences in any first-order theory $T$:
 
 \begin{proof}
Fix a first-order theory $T$. Let $Q$ be a family of variables indexed by the rational numbers. The natural action of $\mathrm{Aut}(\mathbb{Q})$ on $S_Q(T)$, the Stone space of types over $T$ in the variables $Q$, is continuous and so by extreme amenability has a fixed point. A fixed point of this action is precisely the same thing as the type of a $\mathbb{Q}$-indexed indiscernible sequence over $T$, and so we get that there are models of $T$ with indiscernible sequences.
 \end{proof}

 A similar proof of the existence of indiscernible subspaces in Banach theories (Theorem \ref{thm:exist}) is possible, but requires an argument that the analog of $S_Q(T)$ is non-empty (which follows from Dvoretzky's theorem) and also requires more delicate bookkeeping to define the analog of $S_Q(T)$ and to show that the action of the unitary group of a separable Hilbert space is continuous. In the end this is more technical than a proof using Fact \ref{fact:main} directly. 

\section{Indiscernible Subspaces} \label{sec:ind-subsp}

\begin{defn}
 Let $T$ be a Banach {theory}. Let $\mathfrak{M}\models T$ and let $A\subseteq \mathfrak{M}$ be some set of parameters. 
 An \emph{indiscernible subspace over $A$} is a real subspace $V$ of $\mathfrak{M}$ such that for any $n<\omega$ and any $n$-tuples $\bar{b},\bar{c} \in V$, $\bar{b} \equiv_A \bar{c}$ if and only if $\bar{b} \cong \bar{c}$.

{\sloppy If $p$ is a type over $A$, then $V$ is an \emph{indiscernible subspace in $p$ (over $A$)} if it is an indiscernible subspace over $A$ and $b\models p$ for all $b\in V$ with $\left\lVert b \right\rVert = 1$.}


\end{defn}

Note that, as we have defined it, an indiscernible subspace is a real subspace even if $T$ is a complex Banach theory. Also note that an indiscernible subspace in $p$ is not literally contained in the realizations of $p$, but rather has its unit sphere contained in the realizations of $p$. It might be more accurate to talk about ``indiscernible spheres,'' but we find the subspace terminology more familiar.

Indiscernible subspaces are very metrically regular.

\begin{prop}
Suppose $V$ is an indiscernible subspace in some Banach structure. Then $V$ is isometric to a real Hilbert space.

In particular, a real subspace $V$ of a Banach structure is indiscernible over $A$ if and only if it is isometric to a real Hilbert space and for every $n<\omega$ and every pair of $n$-tuples $\bar{b},\bar{c}\in V$, $\bar{b}\equiv_A\bar{c}$ if and only if for all $i,j<n$ $\left<b_i,b_j\right> = \left<c_i,c_j\right>$. 
\end{prop}
\begin{proof}
For any real Banach space $W$, if $\dim W \leq 1$, then $W$ is necessarily isometric to a real Hilbert space. If $\dim V \geq 2$, let $V_0$ be a $2$-dimensional subspace of $V$. A subspace of an indiscernible subspace is automatically an indiscernible subspace, so $V_0$ is indiscernible. For any two distinct unit vectors $a$ and $b$, indiscernibility implies that for any $r,s\in \mathbb{R}$, $\left\lVert r a + s b\right\rVert = \left\lVert s a + r b\right\rVert$, hence the unique linear map that switches $a$ and $b$ fixes $\left\lVert \cdot \right \rVert$. This implies that the automorphism group of $(V_0, \left\lVert \cdot \right \rVert)$ is transitive on the $\left\lVert \cdot \right\rVert$-unit circle. By John's theorem on maximal ellipsoids \cite{MR0030135}, the unit ball of $\left\lVert \cdot \right \rVert$ must be an ellipse, so $\left\lVert \cdot \right \rVert$ is a Euclidean norm.

Thus every $2$-dimensional real subspace of $V$ is Euclidean and so $(V,\left\lVert \cdot \right \rVert)$ satisfies the parallelogram law and is therefore a real Hilbert space.

The `in particular' statement follows from the fact that in a real Hilbert subspace of a Banach space, the polarization identity \cite[Prop.\ 14.1.2]{Blanchard2002} defines the inner product  in terms of a particular quantifier-free formula:
\begin{equation*}
\left<x, y\right> = \frac{1}{4}\left( \left\lVert x + y \right\rVert ^2 - \left\lVert x - y \right\rVert^2 \right).\footnotemark \qedhere
\end{equation*}
\end{proof}
\footnotetext{There is also a polarization identity for the complex inner product: $${\left<x, y\right>_{\mathbb{C}} = \frac{1}{4}\left( \left\lVert x + y \right\rVert ^2 - \left\lVert x - y \right\rVert^2  + i\left\lVert x - iy \right\rVert^2 - i \left\lVert x + iy \right\rVert^2  \right).}$$}

\subsection{Existence of Indiscernible Subspaces}


As mentioned in \cite[Cor.\ 3.9]{SHELAH2019106738}, it follows from Dvoretzky's theorem that if $p$ is a wide type and $\mathfrak{M}$ is a sufficiently saturated model, then $p(\mathfrak{M})$ contains the unit sphere of an infinite dimensional subspace isometric to a Hilbert space. We refine this by showing that, in fact, an indiscernible subspace can be found.  

\begin{thm} \label{thm:exist} 
Let $A$ be a set of parameters in a Banach {theory} $T$ and let $p$ be a wide type over $A$. For any $\kappa$, there is $\mathfrak{M} \models T$ and a subspace $V\subseteq \mathfrak{M}$ of dimension $\kappa$ such that $V$ is an indiscernible subspace in $p$ over $A$. In particular, any $\aleph_0 + \kappa+|A|$-saturated $\mathfrak{M}$ will have such a subspace.
\end{thm}
\begin{proof}
For any set $\Delta$ of $A$-formulas, call a subspace $V$ of a model $\mathfrak{N}$ of $T_A$ \emph{$\Delta$-indiscernible in $p$} if every unit vector in $V$ models $p$ and for any $n<\omega$ and any formula $\varphi \in \Delta$ of arity $n$ and any $n$-tuples $\bar{b},\bar{c} \in V$ with $\bar{b} \cong \bar{c}$, we have $\mathfrak{N}\models \varphi(\bar{b}) = \varphi(\bar{c})$.

Since $p$ is wide, there is a model $\mathfrak{N}\models T$ containing an infinite dimensional subspace $W$ isometric to a real Hilbert space such that for all $b\in W$ with $\left\lVert b \right\rVert = 1$, $b\models p$. This is an infinite dimensional $\varnothing$-indiscernible subspace in $p$.

Now for any finite set of $A$-formulas $\Delta$ and formula $\varphi$, assume that we've shown that there is a model $\mathfrak{N}\models T$ containing an infinite dimensional $\Delta$-indiscernible subspace $V$ in $p$ over $A$. We want to show that there is a $\Delta \cup \{\phi\}$-indiscernible subspace in $V$. By Fact \ref{fact:main}, for every $k<\omega$ there is a $k$-dimensional subspace $W_{k}\subseteq V$ such that for any unit vectors $b_0,\dots,b_{\ell -1},c_0,\dots,c_{\ell-1}$ in $W_{k}$ with $\bar{b}\cong\bar{c}$, we have that $|\varphi^{\mathfrak{N}}(\bar{b})-\varphi^{\mathfrak{N}}(\bar{c})| < 2^{-k}$. If we let $\mathfrak{N}_k = (\mathfrak{N}_k,W_k)$ where we've expanded the language by a fresh predicate symbol $D$ such that $D^{\mathfrak{N}_k}(x)=d(x,W_k)$, then an ultraproduct of the sequence $\mathfrak{N}_k$ will be a structure $(\mathfrak{N}_\omega,W_\omega)$ in which $W_\omega$ is an infinite dimensional Hilbert space.

\emph{Claim:} $W_\omega$ is $\Delta\cup\{\varphi\}$-indiscernible in $p$. 

\emph{Proof of claim.} Fix an $m$-ary formula $\psi \in \Delta \cup \{\varphi\}$ and let $f(k)=0$ if $\psi \in \Delta$ and $f(k)=2^{-k}$ if $\psi = \varphi$. For any $k \geq 2m$, fix $b_0,\dots,b_{m-1},c_0,\dots,c_{m-1}$ in the unit ball of $W_k$, there is a $2m$ dimensional subspace $W^\prime \subseteq W_k$ containing $\bar{b},\bar{c}$. By compactness of $B(W^\prime)^m$ (where $B(X)$ is the unit ball of $X$), we have that for any $\e > 0$ there is a $\delta(\e) > 0$ such that if $|\left<b_i,b_j \right> - \left<c_i,c_j \right>| < \delta(\e)$ for all $i,j < m$ then $|\psi^{\mathfrak{N}}(\bar{b})-\psi^{\mathfrak{N}}(\bar{c})| \leq f(k) + \e$. Note that we can take the function $\delta$ to only depend on $\psi$, specifically its arity and modulus of continuity, and not on $k$, since $B(W^\prime)^m$ is always isometric to $B(\mathbb{R}^{2m})^m$. Therefore, in the ultraproduct we will have $(\forall i,j<m)|\left<b_i,b_j \right> - \left<c_i,c_j \right>| < \delta(\e) \Rightarrow |\psi^{\mathfrak{N}}(\bar{b})-\psi^{\mathfrak{N}}(\bar{c})| \leq \e$ and thus  $\bar{b}\cong \bar{c} \Rightarrow \psi^{\mathfrak{N}_\omega}(\bar{b}) = \psi^{\mathfrak{N}_\omega}(\bar{c})$, as required. \hfill $\qed_{\textit{Claim}}$

Now for each finite set of $A$-formulas we've shown that there's a structure $(\mathfrak{M}_\Delta,V_\Delta)$ (where, again, $V_\Delta$ is the set defined by the new predicate symbol $D$) such that $\mathfrak{M}_\Delta \models T_A$ and $V_\Delta$ is an infinite dimensional $\Delta$-indiscernible subspace in $p$. By taking an ultraproduct with an appropriate ultrafilter we get a structure $(\mathfrak{M},V)$ where $\mathfrak{M}\models T_A$ and $V$ is an infinite dimensional subspace. $V$ is an indiscernible subspace in $p$ over $A$ by the same argument as in the claim.

Finally note that by compactness we can take $V$ to have arbitrarily large dimension and that any subspace of an indiscernible subspace in $p$ over $A$ is an indiscernible subspace in $p$ over $A$, so we get the required result.
\end{proof}

Together with the fact that wide types always exist in Banach theories  with infinite dimensional models \cite[Thm.\ 3.7]{SHELAH2019106738}, we get a corollary. 

\begin{cor} \label{cor:ind-subsp}
Every Banach {theory} with infinite dimensional models has an infinite dimensional indiscernible subspace in some model. In particular, every such theory has an infinite indiscernible set, namely any orthonormal basis of an infinite dimensional indiscernible subspace.
\end{cor}


\section{Minimal{ }Wide Types}



Compare the following Theorem \ref{thm:main} with this fact in discrete logic: If $p$ is a minimal type (i.e.\ $p$ has a unique global non-algebraic extension), then an infinite sequence of realizations of $p$ is a Morley sequence in $p$ if and only if it is an indiscernible sequence.

Here we are using the definition of Morley sequence for (possibly unstable) $A$-invariant types: Let $p$ be a global $A$-invariant type, and let $B\supseteq A$ be some set of parameters. A sequence $\{c_i\}_{i< \kappa}$ is a \emph{Morley sequence in $p$ over $B$} if for all $i< \kappa$, $\mathrm{tp}(c_i/Bc_{<i}) = p \upharpoonright Bc_{<i}$. Note that this definition of Morley sequence agrees with the standard definition for types that are stable in the sense of Lascar and Poizat (as described in \cite[Def.\ 4.1]{SHELAH2019106738}).

\begin{thm} \label{thm:main}
Let $p$ be a minimal{ }wide type over the set $A$. For $\kappa\geq \aleph_0$, a set of realizations $\{b_i\}_{i<\kappa}$ of $p$ is a Morley sequence in (the unique global minimal wide extension of) $p$ if and only if it is an orthonormal basis of an indiscernible subspace in $p$ over $A$. 
\end{thm}
\begin{proof}
All we need to show is that an orthonormal basis of an indiscernible subspace in $p$ over $A$ is a Morley sequence in $p$. The converse will follow from the fact that all Morley sequences in a fixed invariant type of the same length have the same type along with the fact that minimal wide types have a unique global wide extension, which is therefore invariant.

Let $V$ be an indiscernible subspace in $p$ over $A$. Let $\{e_i\}_{i<\kappa}$ be an orthonormal basis of $V$. By construction, $\mathrm{tp}(e_0/A) = p$. Let $q$ be the global minimal wide extension of $p$. Assume that for some $j<\kappa$ we've shown for all $i<j$ that $\mathrm{tp}(e_i/Ae_{<i}) = q \upharpoonright Ae_{<i}$. Let $W = \overline{\mathrm{span}}(e_{\geq j})$. Since $V$ is an indiscernible subspace over $A$, for all unit norm $b,c\in W$, $b\equiv_{Ae_{<j}} c$, so in particular $\mathrm{tp}(b/Ae_{<j})$ is wide. Since $p$ is minimal{ }wide we must have $\mathrm{tp}(b/Ae_{<j}) = q\upharpoonright Ae_{<j}$. Therefore $\{e_i\}_{i<\kappa}$ is a Morley sequence.
\end{proof}

What is unclear at the moment is the answer to this question: 

\begin{quest}
If $p$ is a minimal wide type over the set $A$, is it stable in the sense of \cite[Def.\ 4.1]{SHELAH2019106738}? In other words, is every type $q$ extending $p$ over a model $\mathfrak{M}\supseteq A$  a definable type?
\end{quest}

\section{Counterexamples} \label{sec:count}
Here we collect some counterexamples that may be relevant to any model theoretic development of the ideas presented in this paper.

\subsection{No Infinitary Ramsey-Dvoretzky-Milman Phenomena in General}

Unfortunately some elements of the analogy between the  Ramsey-Dvoretzky-Milman Phenomenon and discrete Ramsey theory do not work. In particular, there is no extension of Dvoretzky's theorem, and therefore Fact \ref{fact:DM-thm}, to $k \geq \omega$, even for a fixed $\e>0$. 
Recall that a linear map $T:X\rightarrow Y$ between Banach spaces is an \emph{isomorphism} if it is a continuous bijection. This is enough to imply that $T$ is invertible and that both $T$ and $T^{-1}$ are Lipschitz. An analog of Dvoretzky's theorem for $k \geq \omega$ would imply that every sufficiently large Banach space has an infinite dimensional subspace isomorphic to Hilbert space, which is known to be false.  Here we will see as specific example of this.

 The following is a well known result in Banach space {theory} (for a proof see the comment after Proposition 2.a.2 in \cite{Lindenstrauss1996}).

\begin{fact} \label{fact:no-no}
For any distinct $X,Y \in \{\ell_p: 1\leq p < \infty\} \cup \{c_0\}$, no subspace of $X$ is isomorphic to $Y$.
\end{fact}

Note that, whereas Corollary \ref{cor:ind-subsp} says that every Banach theory is consistent with the partial type of an indiscernible subspace, the following corollary says that this type can sometimes be omitted in arbitrarily large models (contrast this with the fact that the existence of an Erd\"os cardinal implies that you can find indiscernible sequences in any sufficiently large structure in a countable language \cite[Thm.\ 9.3]{Kanamori2003}).

\begin{cor} \label{cor:no-no-cor}
For $p \in [1,\infty) \setminus \{2\}$, there are arbitrarily large models of $\mathrm{Th}(\ell_p)$ that do not contain any infinite dimensional subspaces isomorphic to a Hilbert space. 
\end{cor}
\begin{proof}
Fix $p \in [1,\infty) \setminus \{2\}$ and $\kappa \geq \aleph_0$.
  Let $\ell_p(\kappa)$ be the Banach space of functions $f:\kappa \rightarrow \mathbb{R}$ such that $\sum_{i<\kappa} |f(i)|^p < \infty$. Note that $\ell_p(\kappa) \equiv \ell_p$.\footnote{To see this, we can find an elementary sub-structure of $\ell_p(\kappa)$ that is isomorphic to $\ell_p$: Let $\mathfrak{L}_0$ be a separable elementary sub-structure of $\ell_p(\kappa)$. For each $i<\omega$, given $\mathfrak{L}_i$, let $B_i$ be the set of all $f \in \ell_p(\kappa)$ that are the indicator function of a singleton $\{i\}$ for some $i$ in the support of some element of $\mathfrak{L}_i$. $B_i$ is countable. Let $\mathfrak{L}_{i+1}$ be a separable elementary sub-structure of $\ell_p(\kappa)$ containing $\mathfrak{L}_i\cup B_i$. $\overline{\bigcup_{i<\omega}\mathfrak{L}_{i+1}}$ is equal to the span of $\bigcup_{i<\omega} B_i$ and so is a separable elementary sub-structure of $\ell_p(\kappa)$ isomorphic to $\ell_p$.} 
 Pick a subspace $V \subseteq \ell_p(\kappa)$. If $V$ is isomorphic to a Hilbert space, then any separable $V_0 \subseteq V$ will also be isomorphic to a Hilbert space. There exists a countable set $A \subseteq \kappa$ such that $V_0 \subseteq \ell_p(A) \subseteq \ell_p(\kappa)$. By Fact \ref{fact:no-no}, $V_0$ is not isomorphic to a Hilbert space, which is a contradiction. Thus no such $V$ can exist.
\end{proof}



Even assuming we start with a Hilbert space we do not get an analog of the infinitary pigeonhole principle (i.e.\ a generalization of Fact \ref{fact:DM-thm}). The discussion by H\'ajeck and Mat\v ej in \cite[after Thm.\ 1]{Hajek2018} of a result of Maurey \cite{Maurey1995} implies that there is a Hilbert theory $T$ with a unary predicate $P$ such that for some $\e>0$ there are arbitrarily large models $\mathfrak{M}$ of $T$ such that for any infinite dimensional subspace $V \subseteq \mathfrak{M}$ there are unit vectors $a,b\in V$ with $|P^{\mathfrak{M}}(a)-P^{\mathfrak{M}}(b)| \geq \e$. 

Stability of a theory often has the effect of making Ramsey phenomena more prevalent in its models, so there is a natural question as to whether anything similar will happen here. Recall that a function $f:S(X)\rightarrow \mathbb{R}$ on the unit sphere $S(X)$ of a Banach space $X$ is \emph{oscillation stable} if for every infinite dimensional subspace $Y \subseteq X$ and every $\e>0$ there is an infinite dimensional subspace $Z \subseteq Y$ such that for any $a,b\in S(Z)$, $|f(a)-f(b)|\leq \e$.

\begin{quest} 
Does (model theoretic) stability imply oscillation stability? That is to say, if $T$ is a stable Banach theory, is every unary formula oscillation stable on models of $T$?
\end{quest}

\subsection{The (Type-)Definability of Indiscernible Subspaces and Complex Banach Structures} \label{subsec:comp}

A central question in the study of inseparably categorical Banach space theories is the degree of definability of the `minimal Hilbert space' that controls a given inseparable model of the theory. Results of Henson and Raynaud in \cite{HensonRaynaud} imply that in general the Hilbert space may not be definable.  In \cite{SHELAH2019106738}, Shelah and Usvyatsov ask whether or not the Hilbert space can be taken to be type-definable or a zeroset. In Example \ref{ex:no-def} we present a simple, but hopefully clarifying, example showing that this is slightly too much to ask.

It is somewhat uncomfortable that even in complex Hilbert structures we are only thinking about \emph{real} indiscernible subspaces rather than \emph{complex} indiscernible subspaces.
In particular, our existing Definition~\ref{defn:main-defn} is incompatible with complex structure:

\begin{prop} \label{prop:no-comp}
Let $T$ be a complex Banach theory. Let $V$ be an indiscernible subspace in some model of $T$. For any non-zero $a\in V$ and $\lambda \in \mathbb{C} \setminus \{0\}$, if $\lambda a \in V$, then $\lambda \in \mathbb{R}$.
\end{prop}
\begin{proof}
Assume that for some non-zero vector $a$, both $a$ and $ia$ are in $V$. We have that $(a,ia)\equiv(ia,a)$, but $(a,ia)\models d(ix,y)=0$ and $(ia,a)\not\models d(ix,y)=0$, which contradicts indiscernibility. Therefore we cannot have that both $a$ and $ia$ are in $V$. The same statement for $a$ and $\lambda a$ with $\lambda \in \mathbb{C}\setminus \mathbb{R}$ follows immediately, since $a,\lambda a \in V \Rightarrow ia \in V$.  
\end{proof}

We could define a notion of \emph{complex indiscernible subspaces} in which types are uniquely determined by (complex-valued) inner products, and that may be appropriate for complex Banach theories, as evidenced by the following. In the case of complex Hilbert space and other Hilbert spaces with a unitary Lie group action, Proposition~\ref{prop:no-comp} is the reason that indiscernible subspaces can fail to be type-definable. We will explicitly give the simplest example of this.



\begin{ex} \label{ex:no-def}
Let $T$ be the theory of an infinite dimensional complex Hilbert space and let $\mathfrak{C}$ be the monster model of $T$. $T$ is inseparably categorical, but for any partial type $\Sigma$ over any small set of parameters $A$, $\Sigma(\mathfrak{C})$ is not an infinite dimensional (real) indiscernible subspace (over $\varnothing$).

\end{ex}
\begin{proof}
$T$ is clearly inseparably categorical by the same reasoning that the theory of real infinite dimensional Hilbert spaces is inseparably categorical (being an infinite dimensional complex Hilbert space is first-order and there is a unique infinite dimensional complex Hilbert space of each infinite density character).  

If $\Sigma(\mathfrak{C})$ is not an infinite dimensional subspace of $\mathfrak{C}$, then we are done, so assume that $\Sigma(\mathfrak{C})$ is an infinite dimensional subspace of $\mathfrak{C}$. Let $\mathfrak{N}$ be a small model containing $A$. Since $\mathfrak{N}$ is a subspace of $\mathfrak{C}$, $\Sigma(\mathfrak{N}) = \Sigma(\mathfrak{C})\cap \mathfrak{N}$ is a subspace of $\mathfrak{N}$. Let $v \in \Sigma(\mathfrak{C})\setminus \Sigma(\mathfrak{N})$. This implies that $v\in \mathfrak{C} \setminus \mathfrak{N}$, so we can write $v$ as $v_\parallel+ v_\perp$, where $v_\parallel$ is the orthogonal projection of $v$ onto $\mathfrak{N}$ and $v_\perp$ is complex orthogonal to $\mathfrak{N}$. Necessarily we have that $v_\perp \neq 0$. Let $\mathfrak{N}^\perp$ be the orthocomplement of $\mathfrak{N}$ in $\mathfrak{C}$. If we write elements of $\mathfrak{C}$ as $(x,y)$ with $x\in \mathfrak{N}$ and $y\in \mathfrak{N}^\perp$, then the maps  $(x,y)\mapsto (x,-y)$, $(x,y)\mapsto (x,iy)$, and $(x,y)\mapsto(x,-iy)$ are automorphisms of  $\mathfrak{C}$ fixing $\mathfrak{N}$. Therefore $(v_\parallel + v_\perp) \equiv_{\mathfrak{N}} (v_\parallel - v_\perp) \equiv_{\mathfrak{N}} (v_\parallel + iv_\perp) \equiv_{\mathfrak{N}} (v_\parallel -iv_\perp)$, so we must have that $(v_\parallel - v_\perp),(v_\parallel + iv_\perp),( v_\parallel - iv_\perp) \in \Sigma(\mathfrak{C})$ as well. Since $\Sigma(\mathfrak{C})$ is a subspace, we have that $b_\perp \in \Sigma(\mathfrak{C})$ and $ib_\perp \in \Sigma(\mathfrak{C})$. Thus by Proposition \ref{prop:no-comp} $\Sigma(\mathfrak{C})$ is not an indiscernible subspace over $\varnothing$.
\end{proof}
This example is a special case of this more general construction: If $G$ is a compact Lie group with an irreducible  unitary representation on $\mathbb{R}^n$ for some $n$ (i.e.\ the group action is transitive on the unit sphere), then we can extend this action to $\ell_2$ by taking the Hilbert space direct sum of countably many copies of the irreducible unitary representation of $G$, and we can think of this as a structure by adding function symbols for the elements of $G$. 
The theory of this structure will be totally categorical and satisfy the conclusion of Example \ref{ex:no-def}. 

Example \ref{ex:no-def} is analogous to the fact that in many strongly minimal theories the set of generic elements in a model is not itself a basis/Morley sequence. The immediate response would be to ask the question of whether or not the unit sphere of the complex linear span (or more generally the `$G$-linear span,' i.e.\ the linear span of $G\cdot V$) of the indiscernible subspace in a minimal{ }wide type agrees with the set of realizations of that minimal{ }wide type, but this can overshoot:

\begin{ex} \label{ex:bad-comp}
Consider the structure whose universe is (the unit ball of) $\ell_2 \oplus \ell_2$ (where we are taking $\ell_2$ as a real Hilbert space), with a complex action $(x,y)\mapsto (-y,x)$ and orthogonal projections $P_0$ and $P_1$ for the sets $\ell_2 \oplus \{\mathbf{0}\}$ and $\{\mathbf{0}\} \oplus \ell_2$, respectively. Let $T$ be the theory of this structure. This is a totally categorical complex Hilbert structure, but for any complete type $p$ and $\mathfrak{M}\models T $, $p(\mathfrak{M})$ does not contain the unit sphere of a non-trivial complex subspace.
\end{ex}
\begin{proof}
$T$ is bi-interpretable with a real Hilbert space, so it is totally categorical. For any complete type $p$, there are unique values of $\left\lVert P_0(x) \right\rVert$ and $\left\lVert P_1(x) \right\rVert$ that are consistent with $p$, so the set of realizations of $p$ in any model cannot contain $\{\lambda a\}_{\lambda \in \mathrm{U}(1)}$ for $a$, a unit vector, and $\mathrm{U}(1) \subset \mathbb{C}$, the set of unit complex numbers. 
\end{proof}
The issue, of course, being that, while we declared by fiat that this is a complex Hilbert structure, the expanded structure does not respect the complex structure.

So, on the one hand, Example \ref{ex:bad-comp} shows that in general the unit sphere of the complex span won't be contained in the minimal{ }wide type. On the other hand, a priori the set of realizations of the minimal{ }wide type could contain more than just the unit sphere of the complex span, such as if we have an $\mathrm{SU}(n)$ action. The complex (or $G$-linear) span of a set is of course part of the algebraic closure of the set in question, so this suggests a small refinement of the original question of Shelah and Usvyatsov:

\begin{quest}
If $T$ is an inseparably categorical Banach {theory}, $p$ is a  minimal{ }wide type, and $\mathfrak{M}$ is a model of $T$ which is prime over an indiscernible subspace $V$ in $p$, does it follow that $p(\mathfrak{M})$ is the unit sphere of a subspace contained in the algebraic closure of $V$? 
\end{quest}

This would be analogous to the statement that if $p$ is a strongly minimal type in an uncountably categorical discrete theory and $\mathfrak{M}$ is a model prime over a Morley sequence $I$ in $p$, then $p(\mathfrak{M})\subseteq \mathrm{acl}(I)$.

\subsection{Non-Minimal Wide Types}

The following example shows, unsurprisingly, that Theorem \ref{thm:main} does not hold for non-minimal{ }wide types. 

\begin{ex}
Let $T$ be the theory of (the unit ball of) the infinite Hilbert space sum $\ell_2 \oplus \ell_2 \oplus \dots$, where we add a predicate $D$ that is the distance to $S^\infty \sqcup S^\infty \sqcup \dots$, where $S^\infty$ is the unit sphere of the corresponding copy of $\ell_2$. This theory is $\omega$-stable. The partial type $\{D = 0\}$ has a unique global non-forking extension $p$ that is wide, but the unit sphere of the linear span of any Morley sequence in $p$ is not contained in $p(\mathfrak{C})$.
\end{ex}
\begin{proof}
This follows from the fact that on $D$ the equivalence relation `$x$ and $y$ are contained in a common unit sphere' is definable by a formula, namely \[E(x,y) = \inf_{z,w \in D}(d(x,z)\dotdiv 1) + (d(z,w)\dotdiv 1) + (d(w,y)\dotdiv 1),\]
where $a \dotdiv b = \max\{a-b,0\}$. If $x,y$ are in the same sphere, then let $S$ be a great circle passing through $x$ and $y$ and choose $z$ and $w$ evenly spaced along the shorter path of $S$. It will always hold that $d(x,z),d(z,w),d(w,y) \leq 1$, so we will have $E(x,y)=0$. On the other hand, if $x$ and $y$ are in different spheres, then $E(x,y)= \sqrt{2} -1$.

Therefore a Morley sequence in $p$ is just any sequence of elements of $D$ which are pairwise non-$E$-equivalent and the unit sphere of the span of any such set is clearly not contained in $D$.
\end{proof}



\bibliographystyle{abbrv}
\bibliography{ref}

%


\end{document}